\documentclass[12pt]{amsart}

\usepackage{amsmath,amsthm,amscd,euscript}
\def \r{\mathbb R}

\def \d{\delta}

\newtheorem{theorem}{Theorem}[section]
\newtheorem{lemma}[theorem]{Lemma}

\newtheorem{proposition}[theorem]{Proposition}
\newtheorem{corollary}[theorem]{Corollary}
\theoremstyle{remark}

\theoremstyle{definition}
\newtheorem{definition}[theorem]{Definition}

\numberwithin{equation}{section}

\title {Bernoulli-Euler numbers and multiboundary singularities of type
$B_n^l$}

\author{Oleg Karpenkov}

\address{Mathematisch Instituut, Universiteit Leiden,
\\ P.O. Box 9512, 2300 RA Leiden, The Netherlands}%
\thanks{The work is partially supported by RFBR NSh-4719.2006.1
grant, by RFBR 05-01-02805-CNRSL\_a grant, NWO-DIAMANT
613.009.001 grant, and by RFBR grant 05-01-01012a.}

\begin{document}
\input epsf
\begin{abstract}
In this paper we study properties of numbers $K_n^l$ of connected
components of bifurcation diagrams for multiboundary
singularities $B_n^l$. These numbers generalize classic
Bernoulli-Euler numbers. We prove a recurrent relation on the
numbers $K_n^l$. As it was known before, $K^1_n$ is $(n{+}1)$-th
Bernoulli-Euler number, this gives us a necessary boundary
condition to calculate $K_n^l$. We also find the generating
functions for $K_n^l$ with small fixed $l$ and write partial
differential equations for the general case. The recurrent
relations lead to numerous relations between Bernoulli-Euler
numbers. We show them in the last section of the paper.
\end{abstract}

\maketitle \tableofcontents

\section{Introduction}

Like binomial coefficients and Fibonacci numbers Bernoulli-Euler
numbers $K_n$ are widely used in many different branches of
mathematics. We define Bernoulli-Euler numbers as Taylor series
coefficients for the function $\sec+\tan$, namely
$$
\sec t+\tan t =\sum \limits_{n=0}^{\infty}K_n\frac{t^n}{n!}.
$$
There exists an equivalent definition of Bernoulli-Euler numbers
by means of a ``classical triangle'', one can construct $K_n$ by
analogy with finding binomial coefficients using Pascal triangle.
We refer the reader to~\cite{Arn21} for the detailed description
of the triangle for Bernoulli-Euler numbers. For arithmetical
properties of Bernoulli-Euler numbers we refer to
papers~\cite{Arn21}, \cite{Knu}, and~\cite{Nie}).

We study one of the geometrical aspects of the numbers $K_n$ in
singularity theory. As it was shown by V.~I.~Arnold
in~\cite{Arn22}, the combinatorics of the components of the set
of very nice M-morsifications for the boundary singularities of
type $B_n$ is closely related to the combinatorics of the
corresponding Springer cones. In~\cite{Arn21} V.~I.~Arnold also
proved that the numbers of connected components coincides with
Bernoulli-Euler numbers.

In this paper we give proofs of theorems announced by author
in~\cite{Ka}. We deal with a generalization of boundary
singularities $B_n$ for the functions on the line to the case of
multiboundary singularities $B_n^l$ with boundaries consisting of
$l$ points. The numbers of connected components of the set of very
nice M-morsifications for $B^l_n$ (denoted by $K_n^l$) is in is
turn a natural generalization of Bernoulli-Euler numbers. In
particular, we prove the recurrent relation on the numbers
$K_n^l$:
$$
K_{n-2}^{l+1}=K_n^l-nlK_n^{l-1}.
$$
Note also that the numbers $K_n^l$ also enumerate certain strata
of singularities of $A_{2l+n-1}$, see Corollary~\ref{strati}.

This work is organized as follows. We start in Section~2 with
necessary notions and definition, in particular we give
definitions of singularities of type $B_k^l$. In Section~3 we
formulate and prove the main theorem on recurrence relation for
numbers $K_n^l$. Further in Section~4 we study the case of small
number of boundary points $l$. We give explicit formulae for the
numbers $K_n^l$ in this case. Finally, in Section~5 we show
general expressions for $K_n^l$ in terms of Bernoulli-Euler
numbers.

The author is grateful to V.~I.~Arnold, V.~M.~Zakalyukin,
S.~K.~Lando, and B.~Z.~Shapiro for useful remarks and help in the
realization of this work, and Mathematisch Instituut of
Universiteit Leiden for the hospitality and excellent working
conditions.

\section{Definition of $B_n^l$ singularities}

Let $N^m$ be a smooth manifold, $U^m$ be a smooth manifold with
smooth boundary, and $\pi :U^m \rightarrow N^m$ --- an immersion,
i.e. a smooth embedding with nondegenerate Jacobian matrix at any
point. We also suppose that the number of preimages of any point
of $N^m$ is bounded from above by some constant. Let $f$ be a
smooth real function on the manifold $N^m$. Denote by $\hat f$
the lifting $f\circ \pi :U^m \rightarrow \r$.

Let the preimage of the point $x_0$ in $N^m$ consists of $l$
boundary points of $U^m$ and any number of nonboundary points of
$U^m$. Let also the collection of $l$ images of tangent
hyperplanes to the boundary points projecting to $x_0$ is a
collection of hyperplanes in general position: the intersection
of any $s$ (for $s \le l$) hyperplanes is a $(m{-}s)$-dimensional
plane.

\begin{definition}\label{d-a}
We say that the function $\hat f$ has a {\it multiboundary
singularity of type $B_n^l$} at the preimage $\pi^{-1}(x_0)$, if
$f$ has a singularity $A_{n-1}$ at point $x_0$. In addition it is
required that the kernel of the Hessian matrix at $x_0$ is
transversal to the projection of any tangent hyperplane at the
boundary point in the preimage $\pi^{-1}(x_0)$.
\end{definition}

Let us remind a definition of a very nice M-morsification of the
boundary singularity $B_{\mu}$ of the paper~\cite{Arn22}.

Consider the space $\r ^{\mu -1}$ of real polynomials with zero
constant turm
$$
x^{\mu}+\lambda _1x^{\mu -1}+ \ldots +\lambda _{\mu -1}x
$$
on the real line with fixed ``boundary'' $x=0$.

\begin{definition}\label{d-a1}
{\it A very nice M-morsification of the boundary singularity
$B_{\mu}$} is a polynomial of this family, whose all $\mu {-}1$
critical points are real and the values at these points are
distinct  and nonzero (i.e. also distinct to the value at the
boundary $x=0$).
\end{definition}
Let us generalize Definition~\ref{d-a1} to the case of
multiboundary singularities of type $B_n^l$.

Consider the Cartesian product of the spaces of real polynomials
($\r ^{n-1}$)
$$
x^n+\lambda _2 x^{n-2}+ \cdots +\lambda _{n-1}x
$$
and the space $\r ^l$ of the space of ``boundary values'' at the
points $x=b_i$, for $i=1,\ldots, l$. We call the points $b_i$ the
{\it boundary points}.

\begin{definition}\label{d-a2}
{\it A very nice M-morsification of the multiboundary singularity
$B_n^l$} is a degree $n$ polynomial with $l$ marked boundary
points, whose all $n{-}1$ critical points are real, and the
values at critical points and at $l$ boundary values are all
pairwise distinct.
\end{definition}

Notice that we enumerate all boundary points, otherwise we should
consider the factor of the space $\r^l$  by the group of all
permutations of the coordinates. We enumerate the boundary points
since they correspond to distinct branches, whose permutations do
not make sense here.

\begin{definition}\label{d-a22}
The {\it M-domain} is the closed subset of the space of all
polynomials $x^n+\lambda_2x^{n-2}+\cdots+ \lambda_{n-1}x$,
consisting of polynomials which critical points are real.
\end{definition}

The set of very nice M-morsifications is an open subset of $\r
^{n-1} \times \r ^l$, and its closure is $(${\it M-domain}$)
\times \r ^l$. The set of very nice M-morsifications splits in
connected components by the bifurcation diagram consisting of
five hypersurfaces. Three of these hypersurfaces come from the
case of boundary singularities of type  $B_n$ (see also in
~\cite{Arn22}):
\begin{description}
\item[(a)] the boundary caustic consisting of functions with a boundary
critical point;
\item[(b)] the ordinary Maxwell stratum consisting of functions with
equal critical values at different points;
\item[(c)] the boundary Maxwell stratum consisting of functions having
some value at the boundary being equal to some critical value (the
corresponding critical point is not at the boundary).
\end{description}
Notice that in our notation $B_n=B_n^1$, moreover the boundary
point does not fixed at zero. So the definition of the very nice
M-morsification of the boundary singularity $B_n$ in the
paper~\cite{Arn22} is a special case of Definition~\ref{d-a2}.

In the case of multiboundary immersion singularities of type
$B_n^l$ for $l\ge 2$ we have additional two hypersurfaces:
\begin{description}
\item[(d)] the double boundary caustic consisting of functions with
some double boundary point;
\item[(e)] the double boundary Maxwell stratum consisting of functions
with equal values at different boundary points.
\end{description}

\section{Recurrent relation on numbers $K_n^l$}

We use $K_n$ to denote Bernoulli-Euler numbers (let us list some
first numbers $K_n$ starting from $n=0$:
$1,1,1,2,5,16,61,\ldots$). Denote by $K_n^l$ the number of
connected components of the set of very nice M-morsifications of
the singularity $B_n^l$. The Bernoulli-Euler numbers are boundary
conditions for the numbers $K_n^l$, namely $K_n^0=K_{n-1}$ and
$K_n^1=K_{n+1}$ (see in~\cite{Arn21}).

\begin{theorem}\label{t-a1}
Let $n\ge 2$ and $l\ge 1$. Then the following identity holds.
$$
K_{n-2}^{l+1}=K_n^l-nlK_n^{l-1}.
$$
\end{theorem}

Consider an example of the singularity of type $B_2^3$. We have
$$
K_3^2=K_5^1-5K_5^0=K_6-5K_4=61-5 \cdot 5=36.
$$
We can calculate $K_3^2$ using recurrence relations in a different
way:
$$
K_3^2=K_1^3+2 \cdot 3K_3^1=3!+6 \cdot 5=36.
$$

We start the proof with two lemmas. In these lemmas we use the
following notation.

Denote by $L_n^l$ the number of connected components of the
boundary caustic in the complement to the union of the strata and
caustics of codimension 2. Here we suppose that the function has
one critical point at the projection of a boundary point, all the
rest critical and boundary points are distinct and the values at
them are also distinct.

\begin{lemma}\label{l-a1}
Let $n\ge 2$ and $l\ge 1$. Then the following identity holds.
$$
L_n^l=l(n-1)K_n^{l-1}.
$$
\end{lemma}

\begin{proof} We take any very nice morsification with $l{-}1$ boundary
points and $n{-}1$ critical points. Consider the action of the
permutation group on the boundary points. This action naturally
defines $(l{-}1)!$ very nice M-morsifications in different
connected components but with the same set of boundary points and
the same polynomial. We put a new $l$-th boundary point to one of
$n{-}1$ critical points. The action of the permutation group on
the boundary points now defines $l!$ distinct M-morsifications
with the same set of boundary points and the same polynomial.
Therefore, we uniquely associate to the collection of $(l-1)!$
old connected components the collection of $(l-1)!$ new connected
components. This implies the statement of Lemma~\ref{l-a1}.
\end{proof}

Further we find the number of connected components of very nice
M-morsifications for which one of the boundary values is either
greater than all the critical values or less than all the critical
values, we denote this number by $\hat K_n^{l}$. Note that we
count twice connected components of all M-morsifications for
which all the critical values are contained in the interval with
endpoints at two boundary values.

\begin{lemma}\label{l-a2}
Let $n\ge 2$ and $l\ge 1$. Then the following identity holds.
$$
\hat K_n^l = 2lK_n^{l-1}
$$
\end{lemma}

\begin{proof} Let $n$ be even and $l\ge 1$. Then any M-morsification has two
branches tending to plus infinity. Consider the maximal boundary
value. If it is greater than all the critical values than it can
be attained only at a point corresponding to two branches
described above. Hence, for any very nice M-morsification with
$l{-}1$ boundary point we have exactly $l$ distinct
M-morsifications for which one of the boundary points $b_i$ where
$1\le i \le l$ is on the right branch, and the value at it is the
maximal among all the boundary and critical values (here we keep
the order of the rest points). The same reasoning is valid for
the boundary point with maximal value at the left branch. Hence,
we obtain $\hat K_{2n}^l = 2lK_{2n}^{l-1}$ for even $n$.

If $n$ is odd, then any M-morsification has two branches, one of
them tends to plus infinity, and the other tends to minus
infinity. By the same reason the statement of Lemma~\ref{l-a2}
hold for odd $n$. We remind that we count connected component
corresponding to the odd case twice.
\end{proof}

{\it Proof of Theorem~\ref{t-a1}}. Consider a connected component
of the set of very nice M-morsifications with $l{+}1$ boundary
points and $n{-}3$ critical points. Take any M-morsifica\-tion of
this component. Consider the last boundary point $b_{l+1}$ and
add to it a $\d$-shaped function concentrated in a small
neighborhood of this point such that the new M-morsification has
a critical point at $b_{l+1}$ with old boundary value. The value
of the second critical point that occurs while adding the
$\d$-shaped function is the maximal among all critical and
boundary values. New M-morsification has now $l$ boundary and
$n{-}1$ critical points. In the same way we can subtract
$\d$-shaped function. This provides us the decomposition of the
set of M-morsifications with $l$ boundary and $n{-}1$ critical
points into couples (corresponding to addition and subtraction of
$\d$-shaped function). For any connected component of the set of
very nice M-morsifications with $l{+}1$ boundary and $n{-}3$
critical points by the above procedure we bijectively associate a
couple of connected components of very nice M-morsifications with
$l$ boundary and $n{-}1$ critical points.

Consider an arbitrary connected component of the set of
M-morsifications of the boundary caustic with $n{-}1$ critical
point and $l$ boundary points. So, one of the boundary points is
critical, and in other boundary and critical points the values
are pairwise distinct. Consider the critical boundary point. Let
us move the boundary point a little to the right or to the left.
If in the critical point we have local maximum than we add
$\d$-shaped function in this critical point, otherwise we
subtract $\d$-shaped function. The critical value becomes then
maximal (or minimal respectively). For the boundary point there
are two positions in correspondence with the direction we have
moved it. This gives us to two very nice M-morsifications with
$n{-}1$ critical points and $l$ boundary points. Hence, we
associate to each of the components of the boundary caustics a
couple of components of very nice M-morsifications with maximal
or minimal critical values.

Finally, consider a component of very nice M-morsifications with
$n{-}1$ critical points and $l$ boundary points with one critical
value greater than all boundary values. Consider an
M-morsification of this component. Let $x_i$ be the point with
maximal critical value. Let us take critical or boundary points
that are the closest to $x_i$ from the left and from the right
respectively. Choose one of this two points with maximal value.
Let us make an operation inverse to the adding $\d$-shaped
function that ``pulls down'' the maximal critical value to the
level of the neighbor chosen point. The result can be of two
types. The first possibility is if we obtain a function with
double critical point. In this case we replace a double point by
a new boundary point $b_{l+1}$. The second possibility is if we
obtain a function having one critical point coinciding with one
boundary point.

We act in the same way for the case of a critical value being
less then a;; boundary values.

The above observations prove the following identity:
$$
2K_n^l - \hat K_n^l= 2L_n^l+2K_n^{l+1}.
$$
Now we apply Lemmas~\ref{l-a1} and~\ref{l-a2}:
$$
2K_n^l -2lK_n^{l-1} =2l(n-1)K_n^{l-1}+2K_{n-2}^{l+1}.
$$
Therefore,
$$
K_{n-2}^{l+1}=K_n^l-nlK_n^{l-1}.
$$
This concludes the proof of Theorem~\ref{t-a1}. \qed

\vspace{2mm}

In conclusion of this section we indicate the relation between
$K_n^l$ and the numbers of connected components of special strata
of bifurcation diagram of critical points and critical values for
degree $2l{+}n$ polynomials (i.e. of $A_{2l+n-1}$).

\begin{corollary}\label{strati}
Consider open strata of caustic of the singularity $A_{2l+n-1}$
corresponding to polynomials of M-domain with $l$-couples of
double critical points with all distinct critical values at
distinct critical points. The number of connected components of
such strata is~${K_n^l/l!}$.
\end{corollary}

\section{Corollaries of Theorem~\ref{t-a1}. Cases of small numbers of boundary points}

Let us give explicit formulae for $K_n^l$ for $l \le 5$.
In~\cite{Arn21} V.~I.~Arnold proved that the number of connected
components of very nice M-morsifications for singularity
$A_{n-1}$ equals to the n-th Bernoulli-Euler number, namely
$K(A_{n-1})=K_{n-1}$. It~\cite{Arn21} there is also a formula for
numbers $K(B_n)$ for multiboundary singularities of type $B_n$:
$K(B_n)=K_{n+1}$. So, we have: $K_n^0=K_{n-1}$, and
$K_n^1=K_{n+1}$.

Direct calculations lead to the results of the following
corollary of Theorem~\ref{t-a1}.
\begin{corollary}\label{c-a1}
The following expressions of $K_n^l$ in terms of Bernoulli-Euler
numbers are correct.
$$
\begin{array}{l}
K_n^2=K_{n+3}-(n+2)K_{n+1};\\
K_n^3=K_{n+5}-(3n+8)K_{n+3};\\
K_n^4=K_{n+7}-(6n+20)K_{n+5}+3(n+2)(n+4)K_{n+3};\\
K_n^5=K_{n+9}-(10n+40)K_{n+7}+(15n^2+110n+184)K_{n+5}\\
\ldots\\
\end{array}
$$
\end{corollary}

The exponential generating function for Bernoulli-Euler numbers is
the function
$$
K(t)=\tan(t)+\sec(t).
$$
Let us show exponential generating functions for the numbers
$B_n^l$ for $l \le 4$. Note that in the cases of $l=0$ and $l=1$
one can propose the exponential generating functions to be $K(t)$,
still in our notation we have
$$
\begin{array}{l}
K_0(t)=\int K(t)dt=
-\ln(\cos(t))+\ln(\tan(\frac{t}{2}+\frac{\pi}{4}))+C \quad \hbox{
and }\\
K_1(t)=K'(t)=\frac{1+K^2}{2}=\frac{1+\sin(t)}{\cos^2{t}}=\frac{1}{1-\sin(t)}\\
\end{array}
$$
respectively.

\begin{corollary}\label{c-a2}
The exponential generating functions for the cases $l=2,3,4$ are
as follows:
$$
\begin{array}{l}
K_2(t)=K'''(t)-(tK(t))''=
\frac{3\sin(t)-t\cos(t)}{(1-\sin(t))^2};\\
K_3(t)=(K''-3tK'+K)'''=
\frac{3}{(1-\sin(t))^3}\Bigl(\sin(t)(3\sin(t)+7)-3t\cos(t)(5+\sin(t))\Bigr);\\
K_4(t)=(K'''-6tK''+(3t^2+4)K'-3tK)^{(4)}=\\
\Bigl(\frac{3t^2}{1-\sin(t)}-\frac{3t\cos(t)}{(1-\sin(t))^2}(3-\sin(t))+
\frac{3(2-\sin(t))}{(1-\sin(t))^2} \Bigr)^{(4)}.\\
\end{array}
$$
\end{corollary}

Let
$$
K(x,y)=\sum\frac{K^l_n}{l!n!}x^ly^n
$$
be an exponential generation function in two variables.
\begin{corollary}\label{c-aa1}
The function $K(x,y)$ satisfies the following differential
equation
$$
 K_x=(1-2x)K_{yy}-xyK_{yyy}.
$$
\end{corollary}

B.~Z.~Shapiro proposed to consider two exponential generating
functions in two variables $R(x,y)$ and $S(x,y)$ separately for
$R_n^l=K_{2n}^l$ and for $S_n^l=K_{2n-1}^l$. This decreases the
order of the differential equation.

\begin{corollary}\label{c-aa2}
The functions $R(x,y)$ and $S(x,y)$ satisfy the following
differential equations:
$$
\begin{array}{c}
R_x=(1-2x)R_{y}-2xyR_{yy},\\
S_x=(1-2x)S_{y}-x(2y-1)S_{yy}.\\
\end{array}
$$
\end{corollary}

Finally, let us describe geometrical structure of $B_n^2$  with
schematic ``pictures'' of fibers in the bundle
$$
\pi :\hbox{\it$($M-domain$)$}\times \r^2 \rightarrow \hbox{\it
M-domain}.
$$
For any polynomial $f(x)$ we associate a skew-symmetric
polynomial $f(b_1)-f(b_2)$ in two variables $b_1$ and $b_2$. This
polynomial defines the double boundary Maxwell stratum and
boundary caustics in the fiber. Boundary caustics and and
boundary Maxwell stratum are defined by vertical and horizontal
straight lines and form a rectangular net, in which the curve
$f(b_1)=f(b_2)$ is inscribed. It is easy to draw combinatoric
pictures of such curves. For example, we draw the pictures for the
singularities $B_3^2$ and $B_4^2$ (see Figures~1 and~2). Notice
that the curve $f(b_1)=f(b_2)$ is a union of a straight line
$b_1=b_2$ with some curve of degree $n{-}1$. For instance, in the
case of $B_3^2$ it is the union of a straight line $b_1=b_2$ and
an ellipse (see Figure 1). Here a natural problem arises. {\it
Describe all combinatoric types of such pictures for general
$n$}. At present moment we do not know the answer to this problem.

\begin{figure}[tbp]
\input tommj.pic
\end{figure}

\section{Corollaries of Theorem~\ref{t-a1}. Connection with Bernoulli-Euler numbers}

In this part we introduce some expressions for the
numbers~$K_n^l$ similar to the expressions of
Corollary~\ref{c-a1}. Further we calculate $K_n^l$ for negative
$n$ satisfying $n\le -l$. These numbers lead to nice relations on
Bernoulli-Euler numbers.

\begin{corollary}\label{c-a3}
The following expression for $K_n^l$ is true:
$$
\begin{array}{l}
K_n^l=K_{n+2l-1}-\\
(\frac{l(l-1)}{2}n+\frac{(l+1)l(l-1)}{3})K_{n+2l-3}+\\
l(l-1)(l-2)(l-3)(\frac{1}{8}n^2+\frac{2l+1}{12}n+\frac{(l+1)(5l-2)}{90})
K_{n+2l-5}+\\
\sum\limits^{[\frac{l}{2}]+1}_{k=4} \Bigl( \Bigl(
\frac{(-1)^{k-1}} {2^{k-1}(k-1)!} \frac{l!}{(l-2k+2)!}
\sum\limits^{k-1}_{d=0}
(p_{k,d}(l)n^{d})\Bigr) K_{n+2(l-k)+1} \Bigr) ,\\
\end{array}
$$
where $p_{k,d}(x)$ is a polynomial of degree $k{-}d{-}1$ with
constant coefficients depending on $k$ and $d$.
\end{corollary}
The idea of the proof is based on the induction on $l$.

Still we do not know the explicit formulae for the polynomials
$p_{k,d}(x)$, nevertheless there exists a recursive method to find
the coefficients of such polynomials. Let us show the polynomials
for $d=k{-}1$, $d=k{-}2$, and $d=k{-}3$.

\begin{corollary}\label{c-a4}
The following hold:
$$
\begin{array}{l}
p_{k,k-1}(x)=1;\\
p_{k,k-2}(x)=\frac{(k-1)}{3}(2x+4-k);\\
p_{k,k-3}(x)=\frac{(k-1)(k-2)}{90}\Bigl(
20x^2+(72-20k)x+(5k^2-39k+64)\Bigr) .
\\
\end{array}
$$
\end{corollary}

Applying formulae of Corollary~\ref{c-a3} one can obtain many
relations for Bernoulli-Euler numbers. Let us consider some
examples of such relations. Substitute $n=1$ and $n=2$ in the
formulae of Corollary~\ref{c-a3}, we get the equalities of the
following theorem.

\begin{theorem}\label{t-a2}
The following relations on Bernoulli-Euler numbers hold:
$$
\begin{array}{l}
K_{2l}-l!=K_{2l}-K_1^l= \sum\limits^{[\frac{l}{2}]+1}_{k=2} \Bigl(
\frac{(-1)^{k}} {2^{k-1}(k-1)!} \frac{l!}{(l-2k+2)!}
\sum\limits^{k-1}_{d=0}
p_{k,d}(l)\Bigr) K_{2(l-k+1)} ;\\

K_{2l+1}-2^{l}l!=K_{2l+1}-K_2^l=
\sum\limits^{[\frac{l}{2}]+1}_{k=2} \Bigl( \frac{(-1)^{k}}
{2^{k-1}(k-1)!} \frac{l!}{(l-2k+2)!} \sum\limits^{k-1}_{d=0}
(p_{k,d}(l)2^{d})\Bigr) K_{2(l-k+1)+1} .\\
\end{array}
$$
\end{theorem}

Let us use a expressions of Corollary~\ref{c-a1} to obtain the
numbers $K_n^l$, for $n\le 0$, and $n\le -l$.

\begin{center}
\begin{tabular}{|c|c|c|c|c|c|}
\hline
$K_n^l$ & l=1 & l=2 & l=3 & l=4 & l=5 \\
\hline
n=0     &  1  &  0  &  0  &  0  &  0  \\
\hline
n=-1    &  1  &  0  &  0  &  0  &  0  \\
\hline
n=-2    &  ?  &  1  &  0  &  0  &  0  \\
\hline
n=-3    &  ?  &  ?  &  2  &  0  &  0  \\
\hline
n=-4    &  ?  &  ?  &  ?  &  6  &  0  \\
\hline
n=-5    &  ?  &  ?  &  ?  &  ?  &  24 \\
\hline
\end{tabular}
\end{center}

The numbers $K_0^l$ can be interpreted as numbers of connected
components of the space of  degree $0$ polynomials with $l$
boundary points with distinct boundary values. Here we have only
one non-zero entry corresponding to $l=1$. The author does not
know any reasonable interpretation for the numbers for negative
$n$.

Notice that there are many zeroes in the table. Let us formulate
this statement more precisely.

\begin{proposition}\label{l-a3}
Let $n \le -1 $, then $K_n^l=0$ for $l > -n$ and
$K_n^{-n}=(-n{-}1)!$. The numbers $K_0^l=0$ for $l>1$.
\end{proposition}

\begin{proof} The proof is based on Theorem~\ref{t-a1}. Let us use
the induction on~$-n$.
$$
\begin{array}{l}
K_0^l=K_2^{l-1}-2(l-1)K_2^{l-2}=(2l-2)!!-(2l-2)(2l-4)!!=0\hbox{
for }l-2 \ge 0.\\
K_{-1}^l=K_1^{l-1}-(l-1)K_2^{l-2}=(l-1)!-(l-1)(l-2)!=0\hbox{ for
}l-2 \ge 0.\\
K_{-2}^l=K_0^{l-1}=0\hbox{ for }l-1 \ge 2.\\
\end{array}
$$
In general case we have
$$
\begin{array}{l}
K_n^l=K_{n+2}^{l-1}-(n+2)(l-1)K_{n+2}^{l-2} \hbox{ for }n<-2, l-2
\ge
-2-n.\\
K_n^{-n}=K_{n+2}^{-n-1}-(n+2)(-n-1)K_{n+2}^{-n-2}=0+(n+2)(n+1)(-n-3)!=(-n-1)!.
\end{array}
$$
This proves the statement of the proposition.
\end{proof}

We conclude the paper with relations for the Bernoulli-Euler
numbers that follows from the results of Corollary~\ref{c-a3} and
Proposition~\ref{l-a3}.

\begin{corollary}\label{t-a3}
Let $n \le 0$, $l>max(1,-n)$, then
$$
0=K_n^l=K_{n+2l-1} +\sum\limits^{[\frac{l}{2}]+1}_{k=2} \Bigl(
\Bigl( \frac{(-1)^{k-1}} {2^{k-1}(k-1)!} \frac{l!}{(l-2k+2)!}
\sum\limits^{k-1}_{d=0} (p_{k,d}(l)n^{d})\Bigr) K_{n+2(l-k)+1}
\Bigr) .
$$
If $l=-n$, then
$$
(l-1)!=K_n^{-n}=K_{l-1} +\sum\limits^{[\frac{l}{2}]+1}_{k=2}
\Bigl( \Bigl( \frac{(-1)^{k-1}} {2^{k-1}(k-1)!}
\frac{l!}{(l-2k+2)!} \sum\limits^{k-1}_{d=0}
(p_{k,d}(l)(-l)^{d})\Bigr) K_{l-2k+1} \Bigr) .
$$
\end{corollary}

In particular for $l>1$ we have:
$$
0=K_0^l=K_{n+2l-1} +\sum\limits^{[\frac{l}{2}]+1}_{k=2} \Bigl(
\Bigl( \frac{(-1)^{k-1}} {2^{k-1}(k-1)!} \frac{l!}{(l-2k+2)!}
p_{k,0}(l)\Bigr) K_{2(l-k)+1} \Bigr) .
$$

\vspace{0.5in}
\end{document}